\documentclass[12pt, reqno]{amsart}
\usepackage{amsmath, amsthm, amscd, amsfonts, amssymb, graphicx, color}
\usepackage[bookmarksnumbered, colorlinks, plainpages]{hyperref}
\hypersetup{colorlinks=true,linkcolor=red, anchorcolor=green, citecolor=cyan, urlcolor=red, filecolor=magenta, pdftoolbar=true}

\newcommand{\beqa}{\begin{eqnarray*}}
\newcommand{\eeqa}{\end{eqnarray*}}
\newcommand{\beqn}{\begin{eqnarray}}
\newcommand{\eeqn}{\end{eqnarray}}

\newcommand{\iy}{\infty}

\newcommand{\lt}{\left}
\newcommand{\rt}{\right}

\newcommand{\C}{\mathbb C}

\newcommand{\R}{\mathbb R}

\newcommand{\N}{\mathbb N}

\newcommand{\Ha}{\mathbb H}

\newcommand{\mcH}{\mathcal H}
\newcommand{\mcB}{\mathcal B}
\newcommand{\mcC}{\mathcal C}

\newcommand{\mcE}{\mathcal E}

\newcommand{\mcJ}{\mathcal J}

\newcommand{\tf}{\tfrac}

\newcommand{\al}{\alpha}

\newcommand{\e}{\varepsilon}

\newcommand{\De}{\Delta}
\newcommand{\la}{\lambda}

\newcommand{\Om}{\Omega}

\newcommand{\s}{\sigma}

\newcommand{\lb}{\label}
\newcommand{\rf}{\ref}
\newcounter{cnt1}
\newcounter{cnt2}
\newcounter{cnt3}
\newcommand{\blr}{\begin{list}{$($\roman{cnt1}$)$}
 {\usecounter{cnt1} \setlength{\topsep}{0pt}
 \setlength{\itemsep}{0pt}}}
\newcommand{\bla}{\begin{list}{$($\alph{cnt2}$)$}
 {\usecounter{cnt2} \setlength{\topsep}{0pt}
 \setlength{\itemsep}{0pt}}}
\newcommand{\bln}{\begin{list}{$($\arabic{cnt3}$)$}
 {\usecounter{cnt3} \setlength{\topsep}{0pt}
 \setlength{\itemsep}{0pt}}}
\newcommand{\el}{\end{list}}
\newtheorem{thm}{Theorem}[section]

\newtheorem{cor}[thm]{Corollary}
\newtheorem{ex}[thm]{Example}

\newtheorem{Def}[thm]{Definition}

\newtheorem{rem}[thm]{Remark}
\newcommand{\Rem}{\begin{rem} \rm}
\newcommand{\bdfn}{\begin{Def} \rm}
\newcommand{\edfn}{\end{Def}}

\newcommand{\ba}{\begin{array}}
\newcommand{\ea}{\end{array}}

\numberwithin{equation}{section}
\date{}
\begin{document}
\setcounter{page}{1}
\title{\bf Some Banach spaces are almost Hilbert}
\author[T. L. Gill and M. Golden]{Tepper L. Gill$^1$$^{*}$ and Marzett Golden$^1$ }
\address{$^{1}$ Department of Mathematics, Howard University,
Washington DC 20059 \\ USA.}
\email{\textcolor[rgb]{0.00,0.00,0.84}{tgill@access4less.net;
tgill@howard.edu}}

\address{$^{2}$ Department of Mathematics, Howard University,
Washington DC 20059 \\ USA.}
\email{\textcolor[rgb]{0.00,0.00,0.84}{nubian83@hotmail.com}}

.


\date{Received: xxxxxx; Revised: yyyyyy; Accepted: zzzzzz.
\newline \indent $^{*}$ Corresponding author}

\begin{abstract}  The purpose of this note is to show that, if $\mcB$ is a uniformly convex Banach, then the dual space $\mcB'$ has a ``  Hilbert space representation" (defined in the paper), that makes $\mcB$  much closer to a Hilbert space then previously suspected. As an application, we prove that, if $\mcB$ also has a Schauder basis (S-basis), then for each $A \in \C[\mcB]$ (the closed and densely defined linear operators), there exists a closed densely defined linear operator $A^* \in \C[\mcB]$ that has all the expected properties of an adjoint.  Thus for example, the bounded linear operators, $L[\mcB]$, is a $^*$algebra.  This result allows us to give a natural definition to the Schatten class of operators on a  uniformly convex Banach space with a S-basis. In particular, every theorem that is true for the Schatten class on a Hilbert space, is also true on such a space.  The main tool we use is a special version of a result due to Kuelbs \cite{K}, which shows that every uniformly convex Banach space with a S-basis can be densely and continuously embedded into a Hilbert space which is unique up to a change of basis.
\end{abstract}
\maketitle
\section{Introduction}
In 1965, Gross \cite{G} proved that every real separable Banach space contains a separable Hilbert space as a dense embedding, and this space is the support of a Gaussian measure.  This was a generalization of Wiener's theory, that was based on the use of the (densely embedded Hilbert) Sobolev space $\Ha^1[0,1] \subset \C[0,1]$. In 1972, Kuelbs \cite{K}
generalized Gross' theorem to include the Hilbert space rigging $\Ha^1[0,1] \subset \C[0,1] \subset L^2[0,1]$.  This general theorem can be stated as:
\begin{thm}{\rm{(Gross-Kuelbs)}}Let $\mcB$ be a separable Banach space. Then there exist separable Hilbert spaces $\mcH_1, \mcH_2$  and a positive trace class operator $T_{12}$ defined on $\mcH_2$ such that $\mcH_1\subset \mcB \subset \mcH_2$ all as continuous dense embeddings, with $\left( {T_{12}^{1/2} u,\;T_{12}^{1/2} v} \right)_1  = \left( {u,\;v} \right)_2$ and  $\left( {T_{12}^{ - 1/2} u,\;T_{12}^{ - 1/2} v} \right)_2  = \left( {u,\;v} \right)_1$.  
\end{thm}
This theorem makes it possible to give a definition of the adjoint for bounded linear operators on separable Banach spaces. The definition has all the expected properties. In particular, It can be shown that, for each bounded linear operator $A$ on $\mcB$, there exists $A^*$, with $A^*A$, maximal accretive, selfadjoint,  $(A^*A)^*=A^*A$, and $I+A^*A$ is invertible (see \cite{GBZS}).

The basic idea is simple, let $A$ be bounded on $\mcB$ and let $A_1$ be the restriction of $A$ to $\mcH_1$.  We can now consider $A_1: \mcH_1 \to \mcH_2$. If ${\bf{J}}_2: \mcH_2 \to \mcH'_2$ is the standard conjugate isomorphism, then $(A'_1){\bf{J}}_2: \mcH_1 \to \mcH'_2$, so that ${\bf{J}}_1^{-1}(A'_1){\bf{J}}_2 : \mcH_1 \to \mcH_1 \subset \mcB$.  It follows that ${\bf{J}}_1^{-1}(A'_1){\bf{J}}_2|_\mcB : \mcB \to \mcB$.  It easy to show that $A^*={\bf{J}}_1^{-1}(A'_1){\bf{J}}_2|_\mcB$ has the the main properties of an adjoint for $A$ on $\mcB$.  

At this level of generality, the definition of an adjoint for closed operators depends on the domain of the operator and changes the choice of $\mcH_1$ (and $T_{12}$) for each operator.  Thus, the adjoint is only reasonable for bounded operators.  It is a open question if all closed densely defined linear operators can have an adjoint with all the expected properties.  Part of the problem is that not all operators in $\mcC[\mcB]$ are of Baire class one (for example, when $\mcB$ is nonreflexive). An operator $A$ is of Baire class one if and only if it can be approximated by a sequence, $\{A_n \}$, of bounded linear operators.  The solution is unknown  and we suspect that, in general, there may be at least one operator in $\mcC[\mcB]$ without an adjoint. 
\subsection{Purpose}
In this note, we focus on uniformly convex Banach spaces, the best class of spaces that are not Hilbert. Our purpose is to show that these spaces are very close to Hilbert spaces and give the best possible results.  In this case, the only difference between the  bounded linear operators $L[\mcB]$ and $L[\mcH]$ is that $L[\mcB]$ is not a $C^*$-algebra. Our main tool is a new representation for the dual space.     We embed $\mcB$ into a (single) Hilbert space $\mcH$ that allows us to define an adjoint $A^*$ on $\mcB$ for each closed densely defined linear operator $A$.  We are also able to define a natural Schatten class structure for  $L[\mcB]$, that is almost identical to the Schatten class on  $\mcH$. 
\subsection{Preliminaries}
The following theorem is due to Lax \cite{L}.
\begin{thm}{\bf{\rm{(Lax's Theorem)}}} Let $\mcB$ be a separable Banach space that is continuously and densely embedded in a Hilbert space $\mcH$ and let $T$ be a bounded linear operator on $\mcB$ that is symmetric with respect to the inner product of $\mcH$ (i.e., $(Tu,v)_\mcH =(u,Tv)_\mcH$ for all $u,v \in \mcB$).  Then:
\begin{enumerate}
\item The operator $T$ is bounded with respect to the $\mcH$ norm and 
\[
\left\| {T^* T} \right\|_\mcH  = \left\| T\right\|_\mcH^2  \leqslant k\left\| T \right\|_\mcB^2,
\] 
where $k$ is a positive constant.
\item  The spectrum of $T$ relative to $\mcH$ is a subset of the spectrum of $T$ relative to $\mcB$.
\item The point spectrum of $T$ relative to $\mcH$ is a equal to the point spectrum of $T$ relative to $\mcB$.
\end{enumerate} 
\end{thm}
\begin{Def}
A family of vectors in a Banach space $\{\mcE_n\} \subset \mcB$, is called a Schauder basis (S-basis) for $\mathcal{B}$ if $\lt\|\mcE_n\rt\|_{\mcB}=1$ and, for each $u \in \mathcal{B}$, there is a unique sequence $(u_n)$ of scalars such that 
\[
u={\rm{lim}}_{k\rightarrow \infty}\sum\limits_{n = 1}^k {u_n \mcE_n } =\sum\limits_{n = 1}^\iy {u_n \mcE_n }. 
\] 
\end{Def}
For example, if $\mcB=L^p[0,1], \; 1<p< \iy$, the family of vectors 
\[
\{1,\cos(2 \pi t), \sin(2 \pi t) \cos(4 \pi t), \sin(4 \pi t), \dots \}
\]
 is a norm one S-basis for $\mcB$.
\begin{Def}
A duality map $\mcJ: \mcB \mapsto \mcB'$, is a set  
\[
\mcJ(u) = \left\{ {{u^*} \in \mcB' \left| {\left\langle {u,{u^*}} \right\rangle  = {{\left\| u \right\|}_\mcB^2}={{\left\| u^* \right\|}_{\mcB'}^2}  { }} \right.} \right\},\;\forall u \in \mcB.
\]
If $\mcB$ is uniformly convex, $\mcJ(u)$ contains a unique functional $u^* \in \mcB'$ for each $u \in \mcB$.
\end{Def}
Let $\Om$ be a bounded open subset of $\R^n, \; n \in \N$. If $u \in L^p[\Om]=\mcB, \; 1<p < \iy$, then the standard example is
\[
u^*=\mcJ(u)(x) = \left\| u \right\|_p^{2-p} \left| {u(x)} \right|^{p - 2} u(x) \in L^q[\Om], \; \tf{1}{p}+\tf{1}{q}=1.
\] 
Furthermore, 
\beqn
\left\langle {u,{u^*}} \right\rangle  = \left\| u \right\|_p^{2-p} \int_{\Om}{\left| {u(x)} \right|^{p}d\la_n(x)} =\left\| u \right\|_p^{2}=\left\| u^* \right\|_q^{2}
\eeqn 
It can be shown that $\mcB$ is uniformly convex and that $u^*=\mcJ(u)$ is uniquely defined for each $u \in \mcB$. Thus, if $\{u_n\}$ is an S-basis for $L^p[\Om]$, then, when normalized, the family vectors $\{u_n^*\}$ is an S-basis for $L^q[\Om]=(L^p[\Om])'$.  The relationship between $u$ and $u^*$ is nonlinear. In the next section we prove the remarkable result, that there is another representation of $\mcB'$, with $u^*={\bf{J}}_\mcB(u)$ linear, for each $u \in \mcB$.  (However,  $u^*$ is no longer a duality mapping.)
\section{The Natural Hilbert space for a Uniformly Convex Banach Space}
In this section we construct the natural Hilbert space for a uniformly convex Banach space with an S-basis. (For this, we only need the Kuelbs part of Theorem 1.1.) Fix $\mcB$ and let $\{\mcE_n\}$ be a S-basis for $\mcB$.     For each $\mcE_n$, let $\mcE_n^*$ be the corresponding dual vector in $\mcB'$ and set $t_n =2^{-n}$.  For each pair $u,v \in \mcB$  define a inner product:
\[
\left( {u,v} \right) = \sum\limits_{n = 1}^\infty  {{t_n}} \left\langle {\mcE_n^*,u} \right\rangle \overline{\left\langle {\mcE_n^*,  v} \right\rangle} 
\]
and let $\mcH$ be the completion of $\mcB$ in the induced norm.  Thus, $\mcB \subset \mcH$ densely and
\beqn\lb{2: one}
{{\left\| u \right\|}_{{{\mcH}}}}={{\left[ \sum\limits_{n=1}^{\infty }{{{t}_{n}}{{\left| \left\langle \mcE_{n}^{*},u\right\rangle  \right|}^{2}}} \right]}^{1/2}}\le \underset{{{{{n}}}}}{\mathop{sup}}\,\lt| \left\langle \mcE_{{n}}^{*},u \right\rangle\rt| \le \underset{{{\left\| \mcE_{{}}^{*} \right\|}_{\mcB'}}\le 1}{\mathop{sup}}\,\lt|\left\langle \mcE_{{}}^{*},u \right\rangle\rt| ={{\left\| u \right\|}_{\mcB}},
\eeqn
so that the embedding is both dense and continuous.  (It is clear that $\mcH$ is unique up to a change of S-basis.)  
\subsection{The Hilbert Space Representation}
In this section, we show that the dual space of a uniformly convex Banach space has a ``Hilbert space representation".
\begin{Def}If $\mcB$ be a Banach space, we say that $\mcB'$ has a Hilbert space representation if there exists a Hilbert space $\mcH$, with $\mcB \subset \mcH$ as a continuous dense embedding and for each $u' \in \mcB', \; u' = (\cdot, u)_\mcH$ for some $u \in \mcB$. 
\end{Def}
\begin{thm}\lb{5: adj} If $\mcB$ be a uniformly convex Banach space with an S-basis, then  $\mcB'$ has a Hilbert space representation.   
\end{thm}
\begin{proof} 
Let $\mcH$ be the natural Hilbert space for $\mcB$ and let ${\bf J}$ be the natural linear mapping from ${\mcH}  \to {\mcH}' $, defined by  
 \[
 \left\langle {v,{\mathbf{J}}(u)} \right\rangle  = {\left( {v,u} \right)_\mcH}, \; {\rm{for\; all}}\; u,v \in \mcH.
 \]   
It is easy to see that ${\bf J}$ is bijective and ${\bf J}^*={\bf J}$.  First, we note that the  restriction  of ${\bf J}$  to $\mcB$, ${\bf J}_\mcB$, maps $\mcB$ to a unique subset of  linear functionals $\{{\bf J}_\mcB(u), \; u \in \mcB\}$ and, ${\bf J}_\mcB(u+v)={\bf J}_\mcB(u)+{\bf J}_\mcB(v)$, for each $u,v \in \mcB$.  We are done if we can prove that  $\{{\bf J}_\mcB(u), \; u \in \mcB\} = \mcB'$.  For this, it suffices to show that ${\bf J}_\mcB(u)$ is bounded for each $u \in \mcB$.  Since  $\mcB$ is dense in $\mcH$, from equation (\rf{2: one}) we have:
\[
{\left\| {{{\mathbf{J}}_{\mcB}}(u)} \right\|_{\mcB'}}=\mathop {\sup }\limits_{v \in \mcB} \frac{{\left\langle {v,{{\mathbf{J}}_\mcB}(u)} \right\rangle }}{{{{\left\| v \right\|}_\mcB}}} \leqslant \mathop {\sup }\limits_{v \in \mcB} \frac{{\left\langle {v,{{\mathbf{J}}_\mcB}(u)} \right\rangle }}{{{{\left\| v \right\|}_\mcH}}} = {\left\| u \right\|_\mcH} \leqslant {\left\| u \right\|_\mcB}.
\]
Thus, $\{{\mathbf{J}}_{\mcB}(u), \; u \in \mcB \} \subset \mcB'$. Since $\mcB$ is  uniformly convex, $\{{\mathbf{J}}_{\mcB}(u), \; u \in \mcB \} = \mcB'$.  
\end{proof} 
\subsection{Construction of the adjoint on $\mcB$}
We can now show that each closed densely linear operator on $\mcB$ has a natural adjoint defined on $\mcB$. 
\begin{thm}\lb{5: adj} Let $\mcB$ be a  uniformly convex Banach space with an S-basis. If $\C[\mcB]$ denotes the closed densely linear operators on $\mcB$ and $L[\mcB]$ denotes the  bounded linear operators, then every $A \in \mcC[\mcB]$ has a well defined adjoint $A^* \in \mcC[\mcB]$. Furthermore, if $A \in L[\mcB]$, then $A^* \in L[\mcB]$ with:  
\begin{enumerate}
\item $(aA)^* ={\bar a} A^*$,
\item $A^{**} =A$,
\item $(A^* +B^*)= A^* + B^*$ 
\item $(AB)^*= B^*A^*$ and 
\item $\lt\|A^*A\rt\|_\mcB \le \lt\|A\rt\|_\mcB^2$.
\end{enumerate}
Thus, $L[\mcB]$ is a $^*$algebra.
\end{thm}
 \begin{proof}
 First, let ${\bf J}$ be the natural linear mapping from ${\mcH}  \to {\mcH}' $ and let ${\bf J}_\mcB$ be the restriction of ${\bf J}$ to $\mcB$.    If ${A}  \in \mcC[{\mcB}]$, then $ {{A'} {\bf J}_{\mcB} } : \mcB' \to \mcB'$. Since $A'$ is closed and densely defined, it follows that ${\bf J}_\mcB^{ - 1} {A'} {\bf J}_\mcB :{\mcB}  \to {\mcB}$ is a closed and densely defined linear operator. We define ${A}^ {*}   = [ {{\bf J}_\mcB^{ - 1} {A'} {\bf J}_\mcB } ] \in \mcC[\mcB]$. If  ${A}  \in L[{\mcB}]$, ${A}^ {*} =  {{\bf J}_\mcB^{ - 1} {A'} {\bf J} }_{\mcB}$ is defined on all of $\mcB$.  By the Closed Graph Theorem,   ${A}^ {*} \in L[\mcB]$.  
The proofs of (1)-(3) are straight forward.  To prove (4),
\beqn\lb{5: prod}
\begin{gathered}
  {\left( {BA} \right)^*} = {\mathbf{J}}_\mcB^{ - 1}{\left( {BA} \right)^\prime }{\mathbf{J}}_\mcB = {\mathbf{J}}_\mcB^{ - 1}{A^\prime }B'{\mathbf{J}}_\mcB \hfill \\
   = \left[ {{\mathbf{J}}_\mcB^{ - 1}{A^\prime }{\mathbf{J}}_\mcB} \right]\left[ {{\mathbf{J}}_\mcB^{ - 1}{B^\prime }{\mathbf{J}}_\mcB} \right] = {A^*}{B^*}. \hfill \\ 
\end{gathered} 
\eeqn
If we replace $B$ by $A^*$ in equation (\rf{5: prod}), noting that $A^{**}=A$, we also see that $(A^*A)^*=A^*A$.  To prove (5), we first see that:
\[
\left\langle {{A^*}Av,{{\mathbf{J}}_\mcB}(u)} \right\rangle  = {\left( {{A^*}Av,u} \right)_\mcH} = {\left( {v,{A^*}Au} \right)_\mcH},
\]
so that $A^*A$ is symmetric.  Thus, by Lax's Theorem, $A^*A$ has a bounded extension to $\mcH$ and $\left\| {A^* A} \right\|_\mcH \leqslant k\left\| {A^* A} \right\|_\mcB $, where $k$ is a positive constant.  We also have that 
\beqn\lb{6: prod}
{\left\| {{A^*}A} \right\|_\mcB} \leqslant {\left\| {{A^*}} \right\|_\mcB}{\left\| A \right\|_\mcB} \leqslant \left\| A \right\|_{_\mcB}^2.
\eeqn
It follows that $\lt\|A^*A\rt\|_\mcB \le \lt\|A\rt\|_\mcB^2$.  If  equality holds in (\rf{6: prod}), for all $A \in L[\mcB]$, then it is a $C^*$-algebra.  It is well-known that this is true if and only if $\mcB$ is a Hilbert space.  Thus, in general  the inequality in (\rf{6: prod}) is strict.
\end{proof}
\subsection{Operators on $\mcB$} 
For the remainder of the paper, we assume that $\mcB$ is uniformly convex and $\mcB'$ carries its Hilbert space representation.
\begin{Def} Let $U$ be bounded, $A \in \C[\mcB]$ and let ${\mathcal{U}},\,{\mathcal{V}}$ be subspaces of ${\mathcal{B}}$.  Then:
\begin{enumerate}
\item $A$ is said to be naturally self-adjoint if $D(A)=D(A^*)$ and $A = A^*$.
\item  $A$ is said to be normal if $D(A)=D(A^*)$ and $AA^*   = A^*  A$.
\item $U$ is unitary if $UU^*   = U^* U= I$.
\item The subspace ${\mathcal{U}}$ is $ \bot $ to ${\mathcal{V}}$ if and only, for each $v \in {\mathcal{V}}$ and $\forall u  \in {\mathcal{U}},\; ( {v , u })_\mcH  = 0$ and, for each $u  \in {\mathcal{U}}$ and  $\forall v  \in {\mathcal{V}}, \; ({u , v })_\mcH  = 0$.
\end{enumerate}
\end{Def}
The last definition is transparent since,  orthogonal subspaces in ${\mathcal{H}}$ induce orthogonal subspaces in ${\mathcal{B}}$.
\begin{thm} (Gram-Schmidt) For each fixed basis $\{ \varphi _i ,\,\,1 \leqslant i < \infty \}$ of ${\mathcal{B}}$, there is at least one set of dual functionals $\{ S_{i} \}$ such that  $\{ \{\psi _i \}, \, \{ S_{i} \},\;1 \leqslant i < \infty \} $ is a biorthonomal set of vectors for ${\mathcal{B}}$,  (i.e., $\left\langle {\psi _i ,S_{j} } \right\rangle  = \delta _{ij}$). 
\end{thm}
\begin{proof}Since each $\varphi _i$ is in ${\mathcal{H}}$, we can construct an orthogonal set of vectors $\{ \phi _i ,\,\,1 \leqslant i < \infty \} $ in ${\mathcal{H}}$ by the standard Gram-Schmidt process.  Set $\psi _i  = {{\phi _i } \mathord{\left/ {\vphantom {{\phi _i } {\left\| {\phi _i } \right\|}}} \right.
 \kern-\nulldelimiterspace} {\left\| {\phi _i } \right\|}}_\mathcal{B} $ and let  $\psi _i^*  = \bf{J}(\psi)$.  From here, it is easy to check that $\{ \{\psi _i \}, \, \{ \psi _i^* \},\;1 \leqslant i < \infty \} $ is a biorthonormal set and the family $\{ \psi _i^* \}$ is unique.  
\end{proof}
We close this section with the following observation about the use of $\mcH$.  Let $A$ be any closed densely defined naturally selfadjoint linear operator on $\mcB$ with a discrete spectrum $\{ \la_i \}$.  It can be extended to ${\mcH}$ with the same properties. If we compute the ratio $\tfrac{{\left\langle {A\psi ,\psi^*  } \right\rangle }}
{{\left\langle {\psi ,\psi^* } \right\rangle }} $ in $\mcB$, it will be ``close" to the value of 
$\tfrac{{\left( {{\bar A} \psi ,\psi } \right)_{{\mcH} } }}
{{\left( {\psi ,\psi } \right)_{{\mcH} } }}$ in ${{\mcH}}$.  By Lax's Theorem \cite{L}, the extension from $\mcB$ to $\mcH$ does not change the point spectrum, so we can use the min-max theorem on $\mcH$ to compute the eigenvalues and eigenfunctions of $A$ via $\bar A$ exactly.   Since $\mcB$ is dense in $\mcH$, it follows that the min-max theorem also holds on $\mcB$.
\subsubsection{Selfadjointness}
With respect to our definition of natural selfadjointness, the following related definition is due to Palmer \cite{P}, where the operator is called symmetric.  This is essentially the same as a Hermitian operator as defined by  Lumer \cite{LU}. (An operator $A$ is dissipative if $-A$ is accretive.)
\begin{Def}
A closed densely defined linear operator $A$ on ${\mathcal{B}}$ is called self-conjugate if both $iA$ and $-iA$ are dissipative.
\end{Def}
\begin{thm}(Vidav-Palmer)  A linear operator $A$, defined on ${\mathcal{B}}$, is self-conjugate if and only if $iA$ and $-iA$ are generators of isometric semigroups.
\end{thm}
\begin{thm}
The operator  $A$, defined on $\mathcal B$, is self-conjugate if and only if it is naturally self-adjoint.
\end{thm}
\begin{proof} Let $\bar A$ and ${\bar A}^*$ be the closed densely defined extensions of $A$ and $A^*$ to ${\mathcal H}$.  On  ${\mathcal{H}}$, $\bar A$ is naturally self-adjoint if and only if $i \bar A$ generates a unitary group, if and only if it is self-conjugate.  Thus, both definitions coincide on ${\mathcal{H}}$.  It follows that the restrictions coincide on ${\mathcal{B}}$.
\end{proof}
The proof of the last theorem represents a general approach for proving new results for $\mcB$.  The following are two representative.
\begin{thm}{\rm{(Polar Representation)}} Let $\mcB$ be a  uniformly convex Banach space with an S-basis. If $A \in \C[\mcB]$, then there exists a partial isometry $U$ and a naturally self-adjoint operator $T$, with $D(T)=D(A)$ and $A=UT$.  Furthermore, $T=[A^*A]^{1/2}$, in a well-defined sense. 
\end{thm}

\begin{thm}{\rm{(Spectral Representation)}} Let $\mcB$ be a  uniformly convex Banach space with an S-basis and let $A \in \C[\mcB]$, be a naturally self-adjoint linear operator.  Then, there exists a operator-valued spectral measure $E_x,  x \in \R$ and for each $u \in D(A)$,  
\[
Au = \int_\mathbb{R} {xd{E_x}(u)}. 
\] 
\end{thm}
\section{{\bf Examples}}
The following related Hilbert space is more general.  It is a concrete implementation of the abstract construction of Kuelbs \cite{K} and, in a different form, is due to Steadman \cite{ST}.   It is constructed over $L^1$, which is not uniformly convex, but is more suitable for applications. It was first used to provide a rigorous foundation for the Feynman path integral formulation of quantum mechanics in \cite{GZ}.  We  use it in this section to provide a few concrete examples.
\subsection{The space $KS^2[\R^n]$}
 
On $\R^n$ let $\mathbb{Q}^n$ be the set $\left\{ {{\mathbf{x}} = (x_1 ,x_2  \cdots ,x_n) \in {\mathbb{R}}^n } \right\}$ such that $
x_i$ is rational for each $i$.  Since this is a countable dense set in ${\mathbb{R}}^n $, we can arrange it as $\mathbb{Q}^n  = \left\{ {{\mathbf{x}}^1, {\mathbf{x}}^2, {\mathbf{x}}^3,  \cdots } \right\}$.  For each $l$ and $i$, let ${\mathbf{B}}_l ({\mathbf{x}}^i ) $ be the closed cube centered at ${\mathbf{x}}^i$,  with sides parallel to the coordinate axes and diagonal $r_l  = 2^{ - l}, l \in {\mathbb{N}}$.  Now choose the  natural order which maps $\mathbb{N} \times \mathbb{N}$ bijectively to $\mathbb{N}$:
\[
\{(1,1), \ (2,1), \ (1,2), \ (1,3), \  (2,2), \  (3,1), \ (3,2), \ (2,3), \  \cdots \}.
\]
Let $\left\{ {{\mathbf{B}}_{k} ,\;k \in \mathbb{N}}\right\}$ be the resulting set  of (all) closed cubes $
\{ {\mathbf{B}}_l ({\mathbf{x}}^i )\;\left| {(l,i) \in \mathbb{N} \times \mathbb{N}\} } \right.
$
centered at a point in $\mathbb{Q}^n $ and let ${\mathcal{E}}_k ({\mathbf{x}})$ be the characteristic function of ${\mathbf{B}}_k $, so that ${\mathcal{E}}_k ({\mathbf{x}})$ is in ${{L}}^p [{\mathbb{R}}^n ]$ for $1 \le p \le \infty$.  Define $F_{k} (\; \cdot \;)$ on $
{{L}}^1 [{\mathbb{R}}^n ] $ by
 \beqn
F_{k} (f)=\int_{{\mathbb{R}}^n }{{\mcE}_{k}({\mathbf{x}})f({\mathbf{x}})d\la_{n}({\mathbf{x}})}. 
\eeqn
It is clear that $F_{k} (\; \cdot \;)$ is a bounded linear functional on ${{L}}^p [{\mathbb{R}}^n ] $ for each ${k}$, $\left\| {F_{k} } \right\|   \le 1$ and, if $F_k (f) = 0$ for all ${k}$, $f = 0$ so that $\left\{ {F_{k} } \right\}$ is fundamental on ${{L}}^p [{\mathbb{R}}^n ] $ for $1 \le p \le \infty$ .
Set ${t_{k}} =2^{-k}$, so that ${\sum\nolimits_{k = 1}^\infty  {t_k}}=1$  and define a measure $d{\mu}$ on ${\mathbb{R}}^n \, \times {\mathbb{R}}^n $ by: 
\[
d{\mu}= \left[ {\sum\nolimits_{k = 1}^\infty  {t_k {\mathcal{E}}_k ({\mathbf{x}}){\mathcal{E}}_k ({\mathbf{y}})} } \right]d\la_{n}({\mathbf{x}})d\la_{n}({\mathbf{y}}).
\]
To construct our Hilbert space, define an inner product $\left( {\; \cdot \;} \right) $ on ${{L}}^1 [{\mathbb{R}}^n ] $ by
\beqn
\begin{gathered}
 \left( {f,g} \right) = \int_{\mathbb{R}^n  \times \mathbb{R}^n } {f({\mathbf{x}})g({\mathbf{y}})^ *  d{\mu}}  \hfill \\
{\text{             }} = \sum\nolimits_{k = 1}^\infty  {t_k } \left[ {\int_{\mathbb{R}^n } {{\mathcal{E}}_k ({\mathbf{x}})f({\mathbf{x}})d\la_{n}({\mathbf{x}})} } \right]\left[ {\int_{\mathbb{R}^n } {{\mathcal{E}}_k ({\mathbf{y}})g({\mathbf{y}})d\la_{n}({\mathbf{y}})} } \right]^ *.   \hfill \\ 
\end{gathered} 
\eeqn
We call the completion of ${{L}}^1 [{\mathbb{R}}^n ] $, with the above inner product, the Kuelbs-Steadman space, ${KS}^2 [{\mathbb{R}}^n ] $.  Steadman \cite{ST} constructed this space by modifying a method developed by Kuelbs \cite{K} for other purposes.  Her interest was in showing that ${{L}}^1 [{\mathbb{R}} ] $ can be densely and continuously embedded in a Hilbert space which contains the non-absolutely integrable functions.  To see that this is the case, suppose $f$ is a non-absolutely integrable function, say Henstock-Kurzweil integral (or of any other type see \cite{H}), then:
\[
\begin{gathered} 
\left\| f \right\|_{{KS}^2}^2  = \sum\nolimits_{k = 1}^\infty  {t_k } \left| {\int_{\mathbb{R}^n } {{\mathcal{E}}_k ({\mathbf{x}})f({\mathbf{x}})d\la_{n}({\mathbf{x}})} } \right|^2  \hfill \\ 
 \leqslant \sup _k \left| {\int_{\mathbb{R}^n } {{\mathcal{E}}_k ({\mathbf{x}})f({\mathbf{x}})d\la_{n}({\mathbf{x}})} } \right|^2  <\iy.  \hfill \\ 
\end{gathered} 
\]
 Since the absolute value is outside the integral, we see that $f \in{KS}^2 [{\mathbb{R}}^n ] $ for any of the definitions of a non-absolute integral (see \cite{GO}).  A detailed discussion of this space and its relationship to the Feynman path integral formulation to quantum mechanics, can be found in \cite{GZ}
\begin{thm} The space ${KS}^2 [{\mathbb{R}}^n ]$ contains ${{L}}^p [{\mathbb{R}}^n ]$ (for each $p,\;1 \leqslant p \leqslant \infty$)  as dense subspaces.
\end{thm}
\begin{proof} By construction, we know  that ${KS}^2 [{\mathbb{R}}^n ]$ contains ${{L}}^1 [{\mathbb{R}}^n ]$ densely.   Thus, we need only show that ${{L}}^q [{\mathbb{R}}^n ]  \subset {KS}^2 [{\mathbb{R}}^n ]$ for $q \ne 1$.  Let $f \in {{L}}^q [{\mathbb{R}}^n ]$ and $q < \infty $. Since  $\left| { {\mathcal{E}}({\mathbf{x}})} \right| = {\mathcal{E}}({\mathbf{x}}) \leqslant 1$ and ${\left| { {\mathcal{E}}({\mathbf{x}})} \right|^q} \leqslant  {\mathcal{E}}({\mathbf{x}})$, we have 
\[
\begin{gathered}
 \left\| f \right\|_{{KS}^2}  = \left[ {\sum\nolimits_{k = 1}^\infty  {t_k } \left| {\int_{{\mathbb{R}}^n } {{\mathcal{E}}_k ({\mathbf{x}})f({\mathbf{x}})d\la_{n}({\mathbf{x}})} } \right|^{\frac{{2q}}{q}} } \right]^{1/2}  \hfill \\
{\text{       }} \leqslant \left[ {\sum\nolimits_{k = 1}^\infty  {t_k } \left( {\int_{{\mathbb{R}}^n } {{\mathcal{E}}_k ({\mathbf{x}})\left| {f({\mathbf{x}})} \right|^q d\la_{n}({\mathbf{x}})} } \right)^{\frac{2}{q}} } \right]^{1/2}  \hfill \\
{\text{      }} \leqslant \sup _k \left( {\int_{{\mathbb{R}}^n } {{\mathcal{E}}_k ({\mathbf{x}})\left| {f({\mathbf{x}})} \right|^q d\la_{n}({\mathbf{x}})} } \right)^{\frac{1}
{q}}  \leqslant \left\| f \right\|_q . \hfill \\ 
\end{gathered} 
\]
Hence, $f \in{KS}^2 [{\mathbb{R}}^n ] $.  For $q = \infty $, first note that $ vol({\mathbf{B}}_k )^2 \le \left[ {\frac{1}
{{2\sqrt n }}} \right]^{2n}$, so we have 
\[
\begin{gathered}
  \left\| f \right\|_{{KS}^2}  = \left[ {\sum\nolimits_{k = 1}^\infty  {t_k } \left| {\int_{{\mathbb{R}}^n } {{\mathcal{E}}_k ({\mathbf{x}})f({\mathbf{x}})d\la_{n}({\mathbf{x}})} } \right|^2 } \right]^{1/2}  \hfill \\
  {\text{       }} \leqslant \left[ {\left[ {\sum\nolimits_{k = 1}^\infty  {t_k [vol({\mathbf{B}}_k )]^2 } } \right][ess\sup \left| f \right|]^2 } \right]^{1/2}  \leqslant {\left[ {\frac{1}
{{2\sqrt n }}} \right]^{n}}\left\| f \right\|_\infty  . \hfill \\ 
\end{gathered} 
\]
Thus $f \in{KS}^2 [{\mathbb{R}}^n ] $, and ${{L}}^\infty  [{\mathbb{R}}^n ] \subset{KS}^2 [{\mathbb{R}}^n ]$.
\end{proof}
The fact that ${{L}}^\infty  [{\mathbb{R}}^n ] \subset{KS}^2 [{\mathbb{R}}^n ]$, while $
{KS}^2 [{\mathbb{R}}^n ] $ is separable makes it clear in a very forceful manner that separability is not an inherited property.  We note that,  since ${{L}}^1 [{\mathbb{R}}^n] \subset {KS}^2 [{\mathbb{R}}^n]$ and ${KS}^2 [{\mathbb{R}}^n]$ is reflexive, the second dual ${{{L}}^1 [{\mathbb{R}}^n]}''  = \mathfrak{M}[{\mathbb{R}}^n] \subset {KS}^2 [{\mathbb{R}}^n]$.   Recall that $\mathfrak{M}[{\mathbb{R}}^n]$ is the space of bounded finitely additive set functions defined on the Borel sets $\mathfrak{B}[{\mathbb{R}}^n]$.  This space contains the Dirac delta measure and free-particle Green's function for the Feynman path integral.

The next result is an unexpected benefit.
\begin{thm} Let $f_n \to f$ weakly in ${{L}^p}, \; 1\le p \le \iy$, then $f_n \to f$ strongly in ${KS}^2$ (i.e., every weakly compact subset of  $L^p$ is compact in ${KS}^2$).
\end{thm}
\begin{proof}
The proof of follows from the fact that, if $\{f_n \}$ is any weakly convergent sequence in $L^p$ with limit $f$, then 
\[
\int_{\mathbb{R}^n } { {\mathcal{E}}_k ({\mathbf{x}})\left[ {f_n ({\mathbf{x}}) - f({\mathbf{x}})} \right]d\la_{n}({\mathbf{x}})}  \to 0
\]
for each $k$.  It follows that $\{f_n \}$ converges strongly to $f$ in ${KS}^2$.  
\end{proof} 

Let $A$ be a closed densely defined linear operator  defined on $L^p[\R^n], \; 1 < p< \iy$,
and let $A'$ be the dual defined on  $L^q[\R^n], \; \tf{1}{p} +\tf{1}{q}=1$.   It is easy to show that, if $A'$ is densely defined on $L^p[\R^n]$, it has a closed extension to $L^p[\R^n]$.  
\begin{ex}Let $A$ be a second order differential operator on $L^p[\R^n]$, of the form
\[
A = \sum\limits_{i,j = 1}^n {{a_{ij}}({\mathbf{x}})} \frac{{{\partial ^2}}}{{\partial {x_i}\partial {x_j}}} + \sum\limits_{i,j = 1}^n { {x_i}{b_{ij}}({\mathbf{x}})}\frac{\partial }{{\partial {x_j}}},
\]
where ${\bf a}({\bf x})=\left[\kern-0.15em\left[ {{a_{ij}}({\mathbf{x}})} 
 \right]\kern-0.15em\right]$ and ${\bf b}({\bf x})=\left[\kern-0.15em\left[ {{b_{ij}}({\mathbf{x}})} \right]\kern-0.15em\right]$ are matrix-valued functions in $\C_c^{\iy}[\R^n \times \R^n]$ (infinitely differentiable functions with compact support).  We also assume that, for all ${\bf x} \in \R^n \; det\left[\kern-0.15em\left[ {{a_{ij}}({\mathbf{x}})} 
 \right]\kern-0.15em\right] > \e$  and the imaginary part of the eigenvalues of ${\bf b}({\bf x})$ are bounded above by $-\e$, for some $\e>0$.
Note, since we don't require  ${\bf a}$ or ${\bf b}$ to be symmetric, $A \ne A'$.

It is well-known that $\C_c^{\iy}[\R^n] \subset L^p[\R^n] \cap L^q[\R^n]$ is dense for all $1 \le p \le q <\iy$.  Furthermore, since $A'$ is invariant on $\C_c^{\iy}[\R^n]$,    
\[
A':\mathbb{C}_c^\infty \left[ {{\mathbb{R}^n}} \right] \subset {L^p}\left[ {{\mathbb{R}^n}} \right] \to \mathbb{C}_c^\infty \left[ {{\mathbb{R}^n}} \right] \subset {L^p}\left[ {{\mathbb{R}^n}} \right].
\]
It follows that  $A'$ has a closed  extension to $L^q[\R^n]$.  (In this case, we do not need $\mcH$ directly, we can identify ${\bf J}$ with the identity on $\mcH$ and $A^*$ with $A'$.)
\end{ex}
\begin{rem}
For a general $A$, which is closed and densely defined on $L^p[\R^n]$, we know that it is densely defined on $KS^2[\R^n]$.  Thus, it has a well-defined adjoint $A^*$ on $KS^2[\R^n]$.  By Theorem \rf{5: adj}, we can take the restriction of $A^*$ from $KS^2[\R^n]$ to obtain our adjoint on $L^q[\R^n]$. 
\end{rem}   
\subsubsection{{\bf Example: Integral Operators}}
In one dimension, the Hilbert transform can be defined on $L^2[\R]$ via its Fourier transform:
\[
\widehat{H(f)} =  - i\operatorname{sgn} x\,\hat{ f}.
\] 
It can also be defined directly as  principal-value integral:
\[
(Hf)(x) = \mathop {\lim }\limits_{\varepsilon  \to 0} \,\frac{1}{\pi }\int_{\left| {x - y} \right| \geqslant \varepsilon } {\frac{{f(y)}}{{x - y}}dy.} 
\]
For a proof of the following results see Grafakos \cite{GRA}, chapter 4.
\begin{thm}  The Hilbert transform on $L^2[\R]$ satisfies:
\begin{enumerate}
\item $H$ is an isometry,  ${\left\| {H(f)} \right\|_2} = {\left\| f \right\|_2}$ and ${H^*} =  - H$.
\item For $f \in L^p[\R], \;1<p<\iy$, there exists a constant $C_p>0$ such that, 
\beqn \lb{5: ht}
{\left\| {H(f)} \right\|_p} \le C_p{\left\| f \right\|_p}.
\eeqn
\end{enumerate}
\end{thm}
The next result is technically obvious, but conceptually non-trivial. 
\begin{cor}The adjoint of $H, \ H^*$ defines a bounded linear operator on $L^p[\R]$ for $1<p <\iy$, and $H^*$ satisfies equation (\rf{5: ht}) for the same constant $C_p$. 
\end{cor}

The Riesz transform, $\bf R$, is the $n$-dimensional analogue of the Hilbert transform and its $j^{\rm th}$ component is defined for $f \in L^p[\R^n], \ 1<p<\iy$, by:
\[
{R_j}(f) = {c_n}\mathop {\lim }\limits_{\varepsilon  \to 0} \int_{\left| {{\mathbf{y}} - {\mathbf{x}}} \right| \geqslant \varepsilon } {\frac{{{y_j} - {x_j}}}{{{{\left| {{\mathbf{y}} - {\mathbf{x}}} \right|}^{n + 1}}}}f({\mathbf{y}})d{\mathbf{y}},\quad {c_n} = } \frac{{\Gamma \left( {\tfrac{{N + 1}}{2}} \right)}}{{{\pi ^{(n + 1)/2}}}}.
\]
\begin{Def} Let $\Om$ be defined on the unit sphere $S^{n-1}$ in $\\R^n$. 
\begin{enumerate}
\item The function $\Om(x)$ is said to be homogeneous of degree $n$ if $\Om(tx)=t^n\Om(x)$.  
\item The function $\Om(x)$ is said to have the cancellation property if
\[
\int_{{S^{n - 1}}} {\Omega ({\mathbf{y}})d\s ({\mathbf{y}}) = 0,\ {\text{where }}d\s {\text{ is the induced Lebesgue measure on }}{S^{n - 1}}.} 
\]
\item The function $\Om(x)$ is said to have the Dini-type condition if
\[
\mathop {\mathop {\sup }\limits_{\left| {{\mathbf{x}} - {\mathbf{y}}} \right| \leqslant \delta } }\limits_{\left| {\mathbf{x}} \right| = \left| {\mathbf{y}} \right| = 1} \left| {\Omega ({\mathbf{x}}) - \Omega ({\mathbf{y}})} \right| \leqslant \omega (\delta ) \Rightarrow \quad \int_0^1 {\frac{{\omega (\delta )d\delta }}{\delta }}  < \infty .
\]
\end{enumerate}
\end{Def}
A proof of the following theorem can be found in Stein \cite{STE} (see pg., 39).
\begin{thm} Suppose that $\Om$ is homogeneous of degree $0$, satisfying both the cancellation property and the Dini-type condition.  If $f \in L^p[\R^n], \ 1<p<\iy$ and
\[
{T_\varepsilon }(f)({\mathbf{x}}) = \int_{\left| {{\mathbf{y}} - {\mathbf{x}}} \right| \geqslant \varepsilon } {\frac{{\Omega ({\mathbf{y}} - {\mathbf{x}})}}{{\left| {{\mathbf{y}} - {\mathbf{x}}} \right|^n}}f({\mathbf{y}})d{\mathbf{y}}.} 
\]
Then
\begin{enumerate}
\item There exists a constant $A_p$, independent of both $f$ and $\varepsilon$ such that
\[
{\left\| {{T_\varepsilon }(f)} \right\|_p} \leqslant {A_p}{\left\| f \right\|_p}.
\] 
\item Furthermore, $\mathop {\lim }\limits_{\varepsilon  \to 0} {T_\varepsilon }(f) = T(f)$ exists in the $L^p$ norm and 
\beqn \lb{5: ht*}
{\left\| {{T}(f)} \right\|_p} \leqslant {A_p}{\left\| f \right\|_p}.
\eeqn 
\end{enumerate}
\end{thm}
Treating $T_\e(f)$ as a special case of the Henstock-Kurzweil integral, conditions (1) and (2) are automatically satisfied and we can write the integral as    
\[
{T }(f)({\mathbf{x}}) = \int_{\R^n } {\frac{{\Omega ({\mathbf{y}} - {\mathbf{x}})}}{{\left| {{\mathbf{y}} - {\mathbf{x}}} \right|^n}}f({\mathbf{y}})d{\mathbf{y}}.} 
\]
For $g \in L^q, \ \tf{1}{p}+ \tf{1}{q}=1$, we have $\left\langle {T(f),g} \right\rangle  = \left\langle {f,{T^*}(g)} \right\rangle$. Using Fubini's Theorem for the Henstock-Kurzweil integral (see \cite{H}), we have that 
\begin{cor} The adjoint of $T, \ T^* =-T$, is defined on $L^p$ and satisfies equation   
(\rf{5: ht*})
\end{cor}
It is easy to see that the Riesz transform is a special case of the above Theorem and Corollary.

Another closely related integral operator is the Riesz potential, ${I_\alpha }(f)({\mathbf{x}})= (-\De)^{-\al/2}f({\mathbf{x}}), \;0 < \alpha  < n$, is defined on $L^p[\R^n], \; 1 <p<\iy$, by (see Stein \cite{STE}, pg., 117):  
\[
{I_\alpha }(f)({\mathbf{x}}) = {\gamma ^{ - 1}}(\alpha )\int_{{\mathbb{R}^n}} {\frac{{f({\mathbf{y}})d{\mathbf{y}}}}{{{{\left| {{\mathbf{x}} - {\mathbf{y}}} \right|}^{n - \alpha }}}}} ,\;{\text{ and }}\gamma {\text{(}}\alpha {\text{) = }}{{\text{2}}^\alpha }{\pi ^{\tfrac{n}{2}}}\frac{{\Gamma (\tfrac{\alpha }{2})}}{{\Gamma (\tfrac{{n - \alpha }}{2})}}.
\]
Since the kernel is symmetric, application of Fubini's Theorem shows that the adjoint $I_{\al}^*=I_{\al}$, is also defined on  $L^p[\R^n]$.  Since $(-\De)^{-1}$ is not bounded, we cannot obtain $L^p$ bounds for  ${I_\alpha }(f)({\mathbf{x}})$.  However, if $1/q=1/p -\al/n$, we have the following (see Stein \cite{STE}, pg., 119)
\begin{thm}If $f \in L^p[\R^n]$ and $0 < \alpha  < n, \; 1<p<q <\iy, \; 1/q=1/p -\al/n$, then the integral defining ${I_\alpha }(f)$ converges absolutely for almost all $\bf x$.  Furthermore, there is a constant $A_{p,q}$, such that 
\beqn
{\left\| {{I}_{\al}(f)} \right\|_q} \leqslant {A_{p,q}}{\left\| f \right\|_p}.
\eeqn 
\end{thm}
\section{Schatten Classes on Banach Spaces} 
In this section, we give a natural definition of the Schatten class of operators on $\mathcal{B}$ (see \cite{SC}) and show that the structure of $L[\mcB]$ is almost identical to that of $L[\mcH]$. 
\subsection{Background: Compact Operators on Banach Spaces} 
Let  $\mathbb{K}(\mathcal{B})$ be the class of compact operators on $\mathcal{B}$ and let $\mathbb{F}(\mathcal{B})$ be the set of operators of finite rank.  Recall that, for separable Banach spaces, $\mathbb{K}({\mathcal{B}})$ is an ideal that need not be the maximal ideal in $L[\mathcal{B}]$. If $\mathbb{M}(\mathcal{B})$ is the set of weakly compact operators and $\mathbb{N}(\mathcal{B})$ is the set of operators that map weakly convergent sequences into strongly convergent sequences, it is known that both are closed two-sided ideals in the operator norm and, in general, $\mathbb{F}(\mathcal{B}) \subset \mathbb{K}(\mathcal{B}) \subset \mathbb{M}(\mathcal{B})$ and $\mathbb{F}(\mathcal{B}) \subset \mathbb{K}(\mathcal{B}) \subset \mathbb{N}(\mathcal{B})$ (see part I of Dunford and Schwartz \cite{DS}, pg. 553).  For reflexive Banach spaces,  $\mathbb{K}(\mathcal{B}) = \mathbb{N}(\mathcal{B})$ and $\mathbb{M}(\mathcal{B}) { = }L[\mathcal{B}]$.  For the space of continuous functions ${\mathbf{C}}[\Omega ]$ on a compact Hausdorff space $\Omega $, Grothendieck \cite{GR} has shown that $\mathbb{M}(\mathcal{B}) { = }\mathbb{N}(\mathcal{B})$.  On the other hand, it is shown in part I of Dunford and Schwartz \cite{DS} that, if $\left( {\Omega ,\Sigma ,\mu } \right)$ a positive measure space, then for ${\mathbf{L}}^1 \left( {\Omega ,\Sigma ,\mu } \right)$ we have $\mathbb{M}(\mathcal{B}) \subset \mathbb{N}(\mathcal{B})$.

We assume that ${\mathcal{B}}$ is uniformly convex, with a S-basis.  In operator theoretic language, our S-basis assumption is that the set of compact operators on $\mcB$ have the approximation property, namely that every compact operator can be approximated by operators of finite rank.  In this section we will show that the  structure of $L[\mcB]$ is almost  identical to its associated space $L[\mcH]$.  The difference is that $L[\mcB]$ is not a $C^*$-algebra (i.e.,  $\left\| {{A^*}A} \right\|_\mcB = {\left\| A \right\|_\mcB^2}, \; A \in L[\mcB]$, is not true for all $A$).   

Let $A$ be a compact operator on ${\mathcal{B}}$ and let $\bar A$ be its extension to $\mathcal{H}$.   For each compact operator $\bar A$, there exists an orthonormal basis $\{ \bar \varphi _n \,\left| {n \geqslant 1} \right.\} $, for $\mathcal{H}$ such that 
\[
\bar A = \sum\nolimits_{n = 1}^\infty  {\mu _n (\bar A)} \left( { \cdot \;,\bar \varphi _n } \right)_2 \bar U\bar \varphi _n.
\]
Where the {$\mu _n $} are the eigenvalues of $[\bar A^*\bar A]^{1/2}  = \left| {\bar A} \right|$, counted by multiplicity and in decreasing order, and $\bar U$ is the partial isometry associated with the polar decomposition of $\bar A = \bar U\left| {\bar A} \right|$.  Without loss, we can assume that the set of functions $\{ \bar \varphi _n \,\left| {n \geqslant 1} \right.\} $ is contained in ${\mathcal{B}}$ and $\{ \varphi _n \,\left| {n \geqslant 1} \right.\} $ is the normalized version in ${\mathcal{B}}$.  If  $\mathbb{S}_p [\mathcal{H} ]$ is the Schatten Class of order $p$ in $L[\mathcal{H} ]$, it is well-known that, if $\bar A \in \mathbb{S}_p [\mathcal{H} ]$, its norm can be represented as:
\[
\begin{gathered}
  \left\| {\bar A} \right\|_p^{{\mcH}} = {\left\{ {Tr{{\left[ {{{\bar A}^*}\bar A} \right]}^{p/2}}} \right\}^{1/p}} = {\left\{ {\sum\limits_{n = 1}^\infty  {\left( {{{\bar A}^*}\bar A{{\bar \varphi }_n},{{\bar \varphi }_n}} \right)_{{\mcH}}^{p/2}} } \right\}^{1/p}} \hfill \\
  {\text{          = }}{\left\{ {\sum\limits_{n = 1}^\infty  {{{\left| {{\mu _n}\left( {\bar A} \right)} \right|}^p}} } \right\}^{1/p}}. \hfill \\ 
\end{gathered} 
\]
\begin{Def} We define the Schatten Class of order $p$ in $L[\mathcal{B}]$ by:
\[
\mathbb{S}_p [\mathcal{B}] = \mathbb{S}_p [\mathcal{H} ]\left| {_\mathcal{B} } \right..
\]
\end{Def}
Since $\bar A$ is the extension of $A \in \mathbb{S}_p [\mathcal{B}]$, we can define $A$ on ${\mathcal{B}}$ by 
\[
A = \sum\nolimits_{n = 1}^\infty  {\mu _n (A)} \left\langle { \cdot \;,{\varphi_n^*}} \right\rangle U\varphi _n, 
\]
where ${\varphi_n^*}$ is the unique functional in $\mcB'$ associated with $\varphi _n $ and  $U$ is the restriction of  $ {\bar U}$ to ${\mathcal{B}}$.  The corresponding norm of $A$ on $\mathbb{S}_p [{\mathcal{B}}]$ is defined by: 
\[
\left\| A \right\|_{_p }^\mathcal{B}  = \left\{ {\sum\nolimits_{n = 1}^\infty  {\left\langle {A^*A\varphi _n , {\varphi_n^*}} \right\rangle ^{p/2} } } \right\}^{1/p}.
 \]
\begin{thm} Let $A \in \mathbb{S}_p [\mathcal{B}]$, then $\left\| A \right\|_{_p }^\mathcal{B}  = \left\| {\bar A} \right\|_{_p }^{\mathcal{H} }$.
\end{thm}
\begin{proof} It is clear that $\{ \varphi _n \,\left| {n \geqslant 1} \right.\} $ is a set of eigenfunctions for $A^* A$ on ${\mathcal{B}}$.  Furthermore, $A^* A$ is naturally selfadjoint and compact, so that its spectrum is discrete.  By Lax's Theorem, the spectrum of $A^* A$ is unchanged by its extension to $\mathcal{H}$.  It follows that $A^* A\varphi _n  = \left| {\mu _n(\bar{A})} \right|^2 \varphi _n $, so that 
\[
\left\langle {A^* A\varphi _n ,{\varphi_n^*}} \right\rangle  = {\left| {\mu _n(A)} \right|^2 }\left\langle { \varphi _n ,\varphi _n^* } \right\rangle  = \left| {\mu_n(\bar{A}) } \right|^2, 
\]
and
\[
\left\| A \right\|_{_p }^\mathcal{B}  = \left\{ {\sum\nolimits_{n = 1}^\infty  {\left\langle {A^* A\varphi _n , {\varphi_n^*}} \right\rangle ^{p/2} } } \right\}^{1/p}  = \left\{ {\sum\nolimits_{n = 1}^\infty  {\left| {\mu _n(\bar{A}) } \right|^p } } \right\}^{1/p}  = \left\| {\bar A} \right\|_{_p }^{\mathcal{H} }.
\]
\end{proof}
It is clear that all of the theory of operator ideals on Hilbert spaces extend to uniformly convex Banach spaces with a S-basis in a straightforward way.  We state a few of the more important results to give a sense of the power provided by the existence of adjoints for spaces of this type.  The first result extends theorems due to Weyl \cite{W}, Horn \cite{HO}, Lalesco \cite{LA} and Lidskii \cite{LI}.  The proofs are all straight forward, for a given $A$ extend it  to $\mcH$, use the Hilbert space result and then restrict back to $\mcB$. 

\begin{thm} Let ${{A}} \in \mathbb{K}({\mathcal{B}})$, the set of compact operators on ${\mathcal{B}}$, and let $\{ {\lambda _n} \}$ be the eigenvalues of ${{A}}$ counted up to algebraic multiplicity.  If $\Phi $ is a mapping on $[0,\infty ]$ which is nonnegative and monotone increasing, then we have:
\begin{enumerate}
\item (Weyl)	
\[
\sum\nolimits_{n = 1}^\iy {\Phi \left( {\left| {\lambda _n ({{A}})} \right|} \right)}  \leqslant \sum\nolimits_{n = 1}^\iy {\Phi \left( {\mu _n ({{A}})} \right)} 
\]
and
\item (Horn)	If ${{A_1}, \ {A_2}} \in \mathbb{K}({\mathcal{B}})$
\[
\sum\nolimits_{n = 1}^\iy {\Phi \left( {\left| {\lambda _n ({{A}}_1 {{A}}_2 )} \right|} \right)}  \leqslant \sum\nolimits_{n = 1}^\iy {\Phi \left( {\mu _n ({{A}}_1 )\mu _n ({{A}}_2 )} \right)}. 
\]
In case ${{A}} \in \mathbb{S}_1 ({\mathcal{B}})$, we have:
\item (Lalesco)	
\[
\sum\nolimits_{n = 1}^\iy {\left| {\lambda _n ({{A}})} \right|}  \leqslant \sum\nolimits_{n = 1}^\iy {\mu _n ({{A}})} 
\]
and
\item (Lidskii)	
\[
\sum\nolimits_{n = 1}^\iy {\lambda _n ({{A}})}  = Tr({{A}}).
\]
\end{enumerate}
\end{thm}
Simon \cite{SI2} provides a very nice approach to infinite determinants and trace class operators on separable Hilbert spaces.  He gives a comparative historical  analysis of Fredholm theory, obtaining a new proof of Lidskii's Theorem as a side benefit and some new insights.  A review of his paper shows that much of it can be directly extended to operator theory on uniformly convex Banach spaces.
\subsection{{Discussion}}
On a Hilbert space ${\mathcal{H}}$, the Schatten classes $\mathbb{S}_p ({\mathcal{H}})$ are the only ideals in $\mathbb{K}({\mathcal{H}})$, and $\mathbb{S}_1 ({\mathcal{H}})$ is minimal.  In a general Banach space, this is far from true.  A complete history of the subject can be found in the recent book by Pietsch \cite{PI1} (see also Retherford \cite{R}, for a nice review).  We  limit this discussion to a few  major topics on the subject.  First, Grothendieck \cite{GR} defined an important class of nuclear operators as follows:

\begin{Def} If ${{A}} \in \mathbb{F}({\mathcal{B}})$ (the operators of finite rank), define the ideal ${\mathbf{N}}_1 ({\mathcal{B}})$
by:
\[
{\mathbf{N}}_1 ({\mathcal{B}}) = \left\{ {{{A}} \in \mathbb{F}({\mathcal{B}})\;\left| {\;{\mathbf{N}}_1 ({{A}}) < \infty } \right.} \right\},
\] 
where
\[
{\mathbf{N}}_1 ({{A}}) = \operatorname{glb} \left\{ {\sum\nolimits_{n = 1}^m {\left\| {f_n } \right\|\left\| {\phi _n } \right\|} \;\left| {f_n  \in {\mathcal{B}'},\;\phi _n  \in {\mathcal{B}},\;{{A}} = \sum\nolimits_{n = 1}^m {\phi _n \left\langle { \cdot \;,\,f_n } \right\rangle } } \right.} \right\}
\]
and the greatest lower bound is over all possible representations for ${{A}}$.
\end{Def}
Grothendieck showed that ${\mathbf{N}}_1 ({\mathcal{B}})$ is the completion of the finite rank operators and is a Banach space with norm ${\mathbf{N}}_1 ( \cdot )$.  It is also a two-sided ideal in $\mathbb{K}({\mathcal{B}})$.   It is easy to show that:

\begin{cor} $\mathbb{M}({\mathcal{B}}),\mathbb{N}({\mathcal{B}})$ and ${\mathbf{N}}_1 ({\mathcal{B}})$ are two-sided *ideals.
\end{cor}
In order to compensate for the (apparent) lack of an adjoint for Banach spaces, Pietsch \cite{PI2}, \cite{PI3}  defined a number of classes of operator ideals for a given ${\mathcal{B}}$.  Of particular importance for our discussion is the class $\mathbb{C}_p ({\mathcal{B}})$, defined by
\[
\mathbb{C}_p ({\mathcal{B}}) = \left\{ {{{A}} \in \mathbb{K}({\mathcal{B}})\;\left| {\,\mathbb{C}_p ({{A}}) = \sum\nolimits_{i = 1}^\infty  {[s_i ({{A}})]^p }  < \infty } \right.} \right\},
\]
where the singular numbers $s_n ({{A}})$ are defined by:
\[
s_n ({{A}}) = \inf \left\{ {\left\| {{{A}} - {{K}}} \right\|_{\mcB} \; \left| \ {{\text{rank of }}{{K}} \leqslant n} \right.} \right\}.
\]
Pietsch has shown that, $\mathbb{C}_1 ({\mathcal{B}}) \subset {\mathbf{N}}_1 ({\mathcal{B}})$, while Johnson et al \cite{JKMR} have shown that for each ${{A}} \in \mathbb{C}_1 ({\mathcal{B}})$, $\sum\nolimits_{n = 1}^\infty  {\left| {\lambda _n ({{A}})} \right|}  < \infty $.  On the other hand, Grothendieck \cite{GR} has provided an example of an operator ${{A}}$ in ${\mathbf{N}}_1 (L^\infty  [0,1])$ with $\sum\nolimits_{n = 1}^\infty  {\left| {\lambda _n ({{A}})} \right|}  = \infty $ (see Simon \cite{SI1}, pg. 118).   Thus, it follows that, in general, the containment is strict.  It is known that, if $\mathbb{C}_1 (\mathcal{B}) = {\mathbf{N}}_1 (\mathcal{B})$, then $\mathcal{B}$ is isomorphic to a Hilbert space (see Johnson et al).  It is clear from the above discussion, that: 
\begin{cor} $\mathbb{C}_p ({\mathcal{B}})$ is a two-sided *ideal in $\mathbb{K}({\mathcal{B}})$, and $\mathbb{S}_1 ({\mathcal{B}}) \subset {\mathbf{N}}_1 ({\mathcal{B}})$.
\end{cor}
For a given Banach space, it is not clear how the spaces $\mathbb{C}_p ({\mathcal{B}})$ of Pietsch relate to our Schatten Classes $\mathbb{S}_p ({\mathcal{B}})$ (clearly $\mathbb{S}_p ({\mathcal{B}}) \subseteq \mathbb{C}_p ({\mathcal{B}})$).  Thus, one question is that of the equality of $\mathbb{S}_p ({\mathcal{B}})$ and $\mathbb{C}_p ({\mathcal{B}})$.  (We suspect that  $\mathbb{S}_1 ({\mathcal{B}})=\mathbb{C}_1 ({\mathcal{B}})$.)
\section{Conclusion}The most interesting aspect of this paper is the observation that the dual space of a Banach space can have more then one representation.  It is well-known that a given Banach space $\mcB$, can have many equivalent norms that generate the same topology.  However, the geometric properties of the space depend on the norm used.  We have shown that the properties of the linear operators on $\mcB$ depend on the family of linear functionals used to represent the dual space $\mcB'$.  This approach offers a interesting tool for a closer study of the structure of bounded linear operators on $\mcB$. 
\bibliographystyle{amsalpha}

\end{document}